\documentclass[11pt]{article}

\usepackage[a4paper,margin=1in]{geometry}
\usepackage{amsmath,amssymb,amsfonts,amsthm,mathtools}
\usepackage{bm}
\usepackage{booktabs}
\usepackage{graphicx}
\usepackage{microtype}
\usepackage{xcolor}
\usepackage{hyperref}

\hypersetup{
  colorlinks=true,
  linkcolor=blue!55!black,
  citecolor=blue!55!black,
  urlcolor=blue!55!black
}

\newtheorem{theorem}{Theorem}[section]
\newtheorem{proposition}[theorem]{Proposition}

\newcommand{\R}{\mathbb R}
\newcommand{\Ftwo}{\mathbb F_2}
\newcommand{\Sone}{\mathbb S^1}
\newcommand{\ii}{\mathrm i}
\newcommand{\dd}{\mathrm d}
\newcommand{\dist}{\operatorname{dist}}
\newcommand{\reach}{\operatorname{reach}}
\newcommand{\Fill}{\operatorname{Fill}}
\DeclareMathOperator*{\argmax}{arg\,max}

\title{Dual-Filtration Topology Recovery from Qualitative Acoustic Scattering Indicators}

\author{
Xiaomei Yang\thanks{School of Mathematics, Southwest Jiaotong University.}
\and
Jiaying Jia\thanks{School of Mathematics, Southwest Jiaotong University.}
\and
Zhiliang Deng\thanks{Corresponding author. School of Mathematical Sciences, University of Electronic Science and Technology of China.}
}

\date{}

\begin{document}
\maketitle

\begin{abstract}
Qualitative inverse scattering methods produce gray-scale indicator fields, from which the topology of the scatterer must be inferred without a prescribed threshold. We propose a dual-filtration method that estimates the component and cavity structure directly from the indicator. Superlevel sets are used to recover the exterior components, while sublevel persistent homology identifies interior cavities associated with the reconstructed envelope. The two types of information are combined into a binary reconstruction, allowing different topological features to be resolved at different intensity levels. We establish geometric and stability results that characterize when the topology can be reliably detected. Full-wave Lippmann--Schwinger experiments with factorization-method indicators show accurate recovery for well-resolved configurations and clear resolution transitions as component separation and cavity size vary.
\end{abstract}

\noindent\textbf{Keywords:}
inverse acoustic scattering; factorization method; persistent homology;
sublevel and superlevel filtrations; Betti numbers; topological resolution.

\section{Introduction}
\label{sec:introduction}

Inverse acoustic scattering seeks to recover geometric or material information
about an unknown scatterer from measurements of the waves generated by incident
fields. Among the available reconstruction techniques, qualitative and
sampling-type methods are attractive because they provide direct imaging
indicators without an explicit parametrization of the unknown boundary.
Classical examples include the linear sampling method
\cite{ColtonKirsch1996,ColtonKress2019} and the factorization method
\cite{Kirsch1998,KirschGrinberg2008}. Related developments include direct
sampling methods based on far-field data \cite{LiZou2013} and extended sampling
methods for more restricted measurement configurations \cite{LiuSun2018}.
Although these methods differ in construction, their numerical output is
typically a scalar indicator whose large values provide evidence for the
presence of the scatterer.

Such indicators are gray-scale rather than binary. A threshold is therefore
needed if one wishes to extract a candidate domain and then describe its
connected components or cavities. Once a threshold has been fixed, ordinary
homology can compute the topology of the corresponding level set. The main
difficulty is the choice of the level itself. The topology may change
substantially with the threshold, and a level that is suitable for one
indicator need not be appropriate for another. Moreover, exterior components
and interior cavities may be represented most clearly at different intensity
levels. A single fixed threshold can therefore be restrictive for scatterers
with several components or nontrivial interior topology.

Persistent homology provides a natural framework for describing such changes
over a family of level sets. Since the introduction of topological persistence
\cite{Edelsbrunner2002} and its algebraic formulation \cite{Zomorodian2005},
it has become a standard tool for extracting multiscale topological
information from data; see also \cite{EdelsbrunnerHarer2010,DeyWang2022}.
Its stability with respect to perturbations of the underlying scalar field
\cite{Cohen2007} is particularly relevant when topology is inferred from
noisy data. Applications of persistent homology to inverse problems remain
comparatively limited. In our earlier work \cite{Deng2026}, for example,
persistent homology was used to construct a topology-sensitive Bayesian prior
for an elliptic inverse problem. Here it plays a different role: topological
information is extracted directly from a qualitative scattering indicator.

Persistent homology has also been used for gray-scale image segmentation and
threshold selection. Chung and Day \cite{ChungDay2018} developed a
persistent-homology approach that estimates segmentation thresholds together
with Betti numbers. The relation between opposite directions of a filtration
has a broader topological basis: extended persistence combines sublevel and
relative-superlevel information \cite{CohenSteiner2009}, while duality
results for digital images relate persistent homology associated with
complementary image constructions \cite{Bleile2022}. Thus,
persistent-homology-based thresholding and the use of complementary
filtrations are established ideas on which the present construction builds.

The additional difficulty in inverse scattering is that the indicator is not
a direct image of the object. It is produced from wave measurements through
the forward model, the observation configuration, the frequency, the noise
level, and the imaging method. Consequently, the topology visible in the
indicator need not coincide with the physical topology of the scatterer. A
sufficiently small cavity may be filled in the indicator, while two nearby
components may appear as a single connected region. Postprocessing cannot
reliably recover topological information that is no longer represented in the
image. We therefore distinguish the physical topology from the
\emph{detectable topology}, by which we mean the component and cavity structure
stably supported by the available indicator.

This distinction motivates the dual-filtration reconstruction proposed here.
The superlevel filtration is used to identify the exterior component structure
of the bright scattering region. Its constant-topology states are ranked by an
integrated boundary-energy criterion, and the selected state provides a filled
exterior envelope. The sublevel filtration is then used to identify persistent
low-intensity components whose birth locations lie inside this envelope. The
component and cavity counts are therefore inferred from the indicator rather
than prescribed in advance, and individual cavities may be reconstructed at
levels different from the level used for the exterior boundary.

To clarify when this topology can be resolved, we also introduce an ideal
distance-profile model. In that model, superlevel sets are geometric offsets
of the scatterer, so positive reach provides a topology-preserving offset
scale. Simple disk and annular geometries then show explicitly how nearby
components merge and small cavities disappear when that scale is exceeded.
For a circular annulus, the persistence of the interior sublevel component can
be expressed directly in terms of the hole radius; persistence stability then
gives a perturbation-dependent sufficient condition for retaining the cavity
bar. The model is not an exact representation of a factorization indicator.
Its purpose is to isolate a geometric resolution mechanism that can be tested
with full-wave data.

The numerical experiments use factorization-method indicators generated from
the full Lippmann--Schwinger model. Representative mixed and multi-hole
configurations test automatic topology recovery without a prescribed Betti
vector. Small-hole and close-component configurations illustrate the
distinction between physical and detectable topology, and parameter sweeps in
component separation and cavity radius reveal the corresponding resolution
transitions.

The remainder of the paper is organized as follows.
Section~\ref{sec:method} introduces the scattering model and the
dual-filtration reconstruction. Section~\ref{sec:analysis} gives the
mathematical interpretation and resolution analysis.
Section~\ref{sec:numerics} presents the full-wave numerical experiments, and
Section~\ref{sec:conclusion} concludes the paper.

\section{Scattering model and dual-filtration reconstruction}
\label{sec:method}

This section defines the full-wave scattering data and the factorization
indicator used in the computations, and then describes the two stages of the
topological reconstruction. The first stage extracts the exterior component
structure from superlevel sets; the second identifies cavities from the
sublevel persistence of the same normalized indicator.

\subsection{Full-wave data and factorization indicator}

Let $D\subset\R^2$ be a bounded penetrable scatterer with contrast
\begin{equation*}
q(x)=q_0\chi_D(x),
\qquad q_0>0,
\end{equation*}
where $\chi_D$ is the characteristic function of $D$. For an incident
direction $d\in\Sone$, the incident plane wave is
\begin{equation*}
u^i(x,d)=\exp(\ii kx\cdot d),
\end{equation*}
where $k>0$ is the wave number. Writing $u=u^i+u^s$ for the total field, with
$u^s$ the scattered field, we have
\begin{equation*}
\Delta u+k^2(1+q(x))u=0
\qquad \text{in }\R^2,
\end{equation*}
and $u^s$ satisfies the Sommerfeld radiation condition at infinity.

With the outgoing free-space Green's function
\begin{equation*}
G_k(x,y)=\frac{\ii}{4}H_0^{(1)}(k|x-y|),
\end{equation*}
the scattering problem is equivalently written as the Lippmann--Schwinger
equation
\begin{equation}
u(x,d)
=
u^i(x,d)
+k^2\int_DG_k(x,y)q(y)u(y,d)\,\dd y,
\qquad x\in\R^2.
\label{eq:LS}
\end{equation}

For an observation direction $\hat x\in\Sone$, the corresponding far-field
pattern is
\begin{equation*}
u^\infty(\hat x,d)
=C_k\int_D\exp(-\ii k\hat x\cdot y)q(y)u(y,d)\,\dd y,
\qquad
C_k=\frac{\exp(\ii\pi/4)k^2}{\sqrt{8\pi k}},
\end{equation*}
under the normalization used here. It defines the far-field operator
$F:L^2(\Sone)\to L^2(\Sone)$ by
\begin{equation*}
(Fg)(\hat x)
=\int_{\Sone}u^\infty(\hat x,d)g(d)\,\dd s(d).
\end{equation*}
Thus $F$ maps a superposition of incident plane waves to the corresponding
far-field pattern and is the data operator used below.

For the factorization indicator, set
\begin{align*}
\Re F&=\frac{F+F^*}{2},
&
\Im F&=\frac{F-F^*}{2\ii},
&
F_\#&=|\Re F|+|\Im F|,
\end{align*}
where $F^*$ is the adjoint of $F$ and $|A|=(A^2)^{1/2}$ for a self-adjoint
operator $A$. This is the standard positive self-adjoint operator used in the
factorization framework; see \cite{Kirsch1998,KirschGrinberg2008,KirschLiu2014}.
Let
\begin{equation*}
F_\#\psi_n=\lambda_n\psi_n,
\qquad \lambda_n\ge0,
\end{equation*}
with $\{\psi_n\}$ orthonormal in $L^2(\Sone)$. For a sampling point
$z\in\Omega_{\rm inv}$, where $\Omega_{\rm inv}\subset\R^2$ is a bounded
imaging window, define
\begin{equation*}
\phi_z(\hat x)=\exp(-\ii k\hat x\cdot z).
\end{equation*}
The regularized factorization-method indicator is
\begin{equation}
I_{\rm FM}(z)
=
\left(
\sum_n
\frac{|\langle\phi_z,\psi_n\rangle_{L^2(\Sone)}|^2}
{\lambda_n+\varepsilon_{\rm FM}}
\right)^{-1},
\qquad z\in\Omega_{\rm inv},
\label{eq:FM}
\end{equation}
where $\varepsilon_{\rm FM}>0$ controls the contribution of very small
eigenvalues. Large values of $I_{\rm FM}$ indicate stronger evidence for the
presence of the scatterer. We use the normalized field
\begin{equation*}
I(z)
=
\frac{I_{\rm FM}(z)-\min_{\Omega_{\rm inv}}I_{\rm FM}}
{\max_{\Omega_{\rm inv}}I_{\rm FM}-\min_{\Omega_{\rm inv}}I_{\rm FM}},
\qquad
I:\Omega_{\rm inv}\to[0,1],
\end{equation*}
for all subsequent topological processing.

In the computations, the incident and observation directions are sampled as
$\{d_j\}_{j=1}^M$ and $\{\hat x_i\}_{i=1}^M$, and the far-field data are
represented by the matrix
\begin{equation*}
F_{ij}=u^\infty(\hat x_i,d_j),
\qquad i,j=1,\ldots,M.
\end{equation*}
For uniformly distributed directions, the common quadrature factor can be
omitted because it only rescales the indicator before normalization. To model
measurement noise, we use
\begin{equation}
F^\delta_{ij}
=
F_{ij}
+
\delta\frac{\max_{p,q}|F_{pq}|}{\sqrt2}
\left(\xi_{ij}^{(1)}+\ii\xi_{ij}^{(2)}\right),
\label{eq:noise}
\end{equation}
where all $\xi_{ij}^{(1)}$ and $\xi_{ij}^{(2)}$ are independent
$\mathcal N(0,1)$ random variables and $\delta\ge0$ is the prescribed relative
noise level. For each noisy realization, the factorization operator,
indicator, and normalization are recomputed from $F^\delta$.

\subsection{Superlevel reconstruction of the exterior structure}
\label{subsec:superlevel}

Let
\begin{equation*}
S_t=\{x\in\Omega_{\rm inv}:I(x)\ge t\},
\qquad 0\le t\le1,
\end{equation*}
be the superlevel sets of the normalized indicator. The imaging window is
chosen large enough that the superlevel sets relevant to the reconstruction do
not meet $\partial\Omega_{\rm inv}$.

For a planar set $A$, let $H_q(A;\Ftwo)$ denote its $q$th homology group with
coefficients in $\Ftwo$, and define
\begin{equation*}
\beta_q(A)=\operatorname{rank}H_q(A;\Ftwo),
\qquad q=0,1.
\end{equation*}
Here $\beta_0$ counts connected components and $\beta_1$ counts holes. We use
\begin{equation*}
\bm\beta(A)=\bigl(\beta_0(A),\beta_1(A)\bigr),
\qquad
\bm\beta^+(t)=\bm\beta(S_t).
\end{equation*}

The interval $[0,1]$ is sampled on an ordered threshold grid. Consecutive
thresholds with the same Betti vector are grouped into maximal blocks. If
$P_j^+=[a_j,b_j]$ denotes the interval spanned by the $j$th block, then
\begin{equation*}
\bm\beta^+(t)=\bm\beta_j^+
\qquad
\text{at all sampled levels }t\in P_j^+.
\end{equation*}
We refer to $P_j^+$ as a \emph{superlevel topological state}.

The width of a state records how long a topology persists as the threshold
varies, but it does not measure how strongly the corresponding boundary is
represented by the indicator. We therefore rank the states using the gradient
carried by their level-set boundaries. For a regular level $t$, set
\begin{align*}
B_I^+(t)
&=
\int_{\partial S_t}|\nabla I(x)|\,\dd\mathcal H^1(x),
\\
\mathcal E_I^+(P_j^+)
&=
\int_{a_j}^{b_j}B_I^+(t)\,\dd t,
\end{align*}
where $\mathcal H^1$ is the one-dimensional Hausdorff measure. The selected
state is
\begin{equation*}
j_+^*\in
\argmax_{1\le j\le J}\mathcal E_I^+(P_j^+).
\end{equation*}

The superlevel stage is used only to estimate the exterior component count, so
we retain
\begin{equation*}
\widehat\beta_0=\beta_0(S_t),
\qquad t\in P_{j_+^*}^+.
\end{equation*}
The full Betti vector is nevertheless used to form the states. In this way, a
temporary hole-creation or hole-filling event is not merged with a neighboring
state that has the same number of connected components. Cavity counting is
deferred to the sublevel stage.

Within the selected state, we choose the exterior boundary level and envelope
by
\begin{align*}
t_+^*
&\in
\argmax_{t\in P_{j_+^*}^+}B_I^+(t),
\\
E_*
&=
\Fill(S_{t_+^*}).
\end{align*}
Here $\Fill(A)$ is obtained by filling every connected component of
$\Omega_{\rm inv}\setminus A$ that does not meet
$\partial\Omega_{\rm inv}$. Thus $E_*$ preserves the selected exterior
components while temporarily suppressing their interior holes. It serves as
the spatial envelope for cavity detection.

\subsection{Sublevel cavity detection and final reconstruction}
\label{subsec:sublevel}

Cavities appear as low-intensity regions surrounded by the bright material
support. We therefore consider the sublevel filtration
\begin{equation*}
L_t=\{x\in\Omega_{\rm inv}:I(x)\le t\},
\qquad 0\le t\le1,
\end{equation*}
and compute its finite zero-dimensional persistence bars
\begin{equation*}
[b_r,d_r),
\qquad
p_r=d_r-b_r.
\end{equation*}
For each finite bar, let $x_r^{\rm b}\in\Omega_{\rm inv}$ denote the grid
point at which the corresponding connected component is born in the lower-star
filtration. Essential $H_0$ classes are not used as cavity candidates.

The envelope $E_*$ separates interior minima from the exterior background.
Given a persistence cutoff $\tau>0$, define the cavity index set
\begin{equation}
\mathcal C_\tau(E_*)
=
\left\{
r:
 x_r^{\rm b}\in\operatorname{int}(E_*),
 \quad
 p_r\ge\tau
\right\}.
\label{eq:cavity-index-set}
\end{equation}
The detectable cavity count and Betti vector are then
\begin{equation*}
\widehat\beta_1=\#\mathcal C_\tau(E_*),
\qquad
\widehat{\bm\beta}_{\rm det}
=(\widehat\beta_0,\widehat\beta_1).
\end{equation*}
The birth-location condition in \eqref{eq:cavity-index-set} prevents persistent
low-intensity components originating in the exterior background from being
interpreted as cavities.

Persistence determines which cavities are retained; their geometry is
reconstructed separately. For $r\in\mathcal C_\tau(E_*)$ and
$b_r<t<d_r$, let $C_r(t)$ be the connected component of $L_t\cap E_*$ that
contains $x_r^{\rm b}$. We define the mean boundary-gradient score
\begin{equation*}
G_r(t)
=
\frac{1}{\mathcal H^1(\partial C_r(t))}
\int_{\partial C_r(t)}|\nabla I(x)|\,\dd\mathcal H^1(x)
\end{equation*}
when the boundary has positive length, and choose
\begin{equation*}
t_r^-
\in
\argmax_{b_r<t<d_r}G_r(t),
\qquad
\widehat C_r=C_r(t_r^-).
\end{equation*}
The final reconstruction is
\begin{equation*}
\widehat D
=
E_*
\setminus
\bigcup_{r\in\mathcal C_\tau(E_*)}\widehat C_r.
\end{equation*}
Thus the exterior envelope and the individual cavities are not forced to share
a single global threshold. The true Betti vector of $D$ is never supplied to
the algorithm.

\section{Mathematical interpretation and topological resolution}
\label{sec:analysis}

This section explains the mathematical structure behind the two stages of the
algorithm and then introduces an idealized model for its resolution limits.
Planar duality relates cavities to connected components of a complement, the
coarea formula interprets the superlevel state score, and a distance-profile
model makes the geometric scales of component merging and cavity filling
explicit.

\subsection{Planar duality and the superlevel energy}

Let $\widetilde H_q$ and $\widetilde H^q$ denote reduced homology and reduced
cohomology, respectively.

\begin{proposition}
\label{prop:alexander}
Let $K\subset\R^2$ be compact and locally contractible. Then
\begin{equation*}
\widetilde H_1(K;\Ftwo)
\cong
\widetilde H^0(S^2\setminus K;\Ftwo),
\end{equation*}
where $S^2$ is the one-point compactification of $\R^2$. Consequently,
\begin{equation*}
\beta_1(K)
=
N_{\rm bnd}(\R^2\setminus K),
\end{equation*}
where $N_{\rm bnd}$ denotes the number of bounded connected components of the
complement.
\end{proposition}

\begin{proof}
Alexander duality gives
\begin{equation*}
\widetilde H_q(K;\Ftwo)
\cong
\widetilde H^{1-q}(S^2\setminus K;\Ftwo).
\end{equation*}
For $q=1$, the dimension of the right-hand side is one less than the number of
connected components of $S^2\setminus K$. Exactly one component contains the
point at infinity; all remaining components are bounded complementary
components in $\R^2$.
\end{proof}

This fixed-level statement explains the spatial interpretation of the cavity
stage: holes of a bright planar set correspond to bounded components of its
dark complement. The proposed method does not merely compute the same
information at a single threshold. Instead, it follows low-intensity
components across their own filtration and uses the exterior envelope to
select those born inside the reconstructed object.

\begin{proposition}
\label{prop:coarea}
Let $I$ be Lipschitz on $\Omega_{\rm inv}$ and let $a<b$. Then
\begin{equation*}
\int_a^b
\int_{I^{-1}(t)}|\nabla I|\,\dd\mathcal H^1\,\dd t
=
\int_{\{a<I<b\}}|\nabla I(x)|^2\,\dd x.
\end{equation*}
If $P_j^+=[a_j,b_j]$ is a superlevel state whose level sets do not meet
$\partial\Omega_{\rm inv}$, then, for regular levels almost everywhere,
\begin{equation*}
\mathcal E_I^+(P_j^+)
=
\int_{\{a_j<I<b_j\}}|\nabla I(x)|^2\,\dd x.
\end{equation*}
\end{proposition}

\begin{proof}
Apply the coarea formula to
$|\nabla I|\bm 1_{\{a<I<b\}}$. For regular values of $t$,
$I^{-1}(t)=\partial S_t$, which gives the second identity.
\end{proof}

Thus $\mathcal E_I^+(P_j^+)$ is the Dirichlet energy carried by the indicator
through the intensity slab associated with one topological state. This gives a
geometric interpretation of the ranking criterion. It does not, however,
imply that the highest-energy state must coincide with the physical topology
for an arbitrary imaging operator.

\subsection{An ideal model for topological resolution}

For a closed set $D\subset\R^2$, let
\begin{equation*}
\rho_D(x)=\dist(x,D),
\qquad
D_r=\{x\in\R^2:\rho_D(x)\le r\}.
\end{equation*}
The reach of $D$, introduced by Federer \cite{Federer1959}, is
\begin{equation*}
\reach(D)
=
\sup\left\{
\rho>0:
\text{every }x\text{ with }\rho_D(x)<\rho
\text{ has a unique nearest point in }D
\right\}.
\end{equation*}

To model spatial blurring, consider
\begin{equation*}
I_0(x)=\Phi(\rho_D(x)),
\qquad x\in\Omega_{\rm inv},
\end{equation*}
where $\Phi$ is continuous and strictly decreasing on the relevant distance
range and satisfies $\Phi(0)=1$. If $r(t)=\Phi^{-1}(t)$, then
\begin{equation*}
\{I_0\ge t\}=D_{r(t)}\cap\Omega_{\rm inv}.
\end{equation*}
Thus, whenever the offset is contained in the imaging window, the superlevel
sets of the ideal indicator are precisely the geometric offsets of $D$.

\begin{theorem}
\label{thm:reach}
Assume $\reach(D)>0$. If $0\le r<\reach(D)$, then $D_r$ deformation
retracts onto $D$. Consequently,
\begin{equation*}
H_q(D_r;\Ftwo)\cong H_q(D;\Ftwo),
\qquad
\beta_q(D_r)=\beta_q(D),
\qquad q=0,1.
\end{equation*}
\end{theorem}

\begin{proof}
For $r<\reach(D)$, the nearest-point projection
$\pi_D:D_r\to D$ is well defined and continuous. Define
\begin{equation*}
H(x,s)=(1-s)x+s\pi_D(x),
\qquad 0\le s\le1.
\end{equation*}
For $x\in D_r$,
\begin{equation*}
\dist(H(x,s),D)
\le |H(x,s)-\pi_D(x)|
=(1-s)|x-\pi_D(x)|
\le r,
\end{equation*}
so the homotopy remains inside $D_r$. It fixes $D$ pointwise and retracts
$D_r$ onto $D$.
\end{proof}

The theorem identifies a topology-preserving scale in the ideal model. Two
simple geometries make the corresponding failure mechanisms explicit. If $D$
is the union of two disjoint closed disks whose boundaries are separated by a
gap $g>0$, then
\begin{equation*}
\reach(D)=\frac{g}{2}.
\end{equation*}
The two offsets remain disjoint for $r<g/2$ and first touch at $r=g/2$.
Thus a small separation is naturally vulnerable to merging under an
offset-like indicator.

For the annulus
\begin{equation*}
D_{a,b}=\{x\in\R^2:a\le|x|\le b\},
\qquad 0<a<b,
\end{equation*}
one has $\reach(D_{a,b})=a$. For $0\le r<a$,
\begin{equation*}
(D_{a,b})_r=\{x:a-r\le|x|\le b+r\},
\end{equation*}
whereas the central cavity is filled when $r\ge a$. Hence the inner radius is
the exact offset scale at which the annular topology is lost.

These formulas are not used by the reconstruction algorithm. They provide a
reference model for interpreting the component-separation and cavity-size
experiments in Section~\ref{sec:numerics}.

\subsection{Cavity persistence and perturbation stability}

The ideal model also gives an explicit relation between cavity size and
sublevel persistence.

\begin{proposition}
\label{prop:annuluspersistence}
Let $D=D_{a,b}$ be a circular annulus and let
$I_0(x)=\Phi(\dist(x,D))$ as above. Assume that $\Omega_{\rm inv}$ contains
an exterior point whose distance from $D$ is greater than $a$. Then the cavity
component in the global sublevel filtration of $I_0$ gives a finite $H_0$ bar
with
\begin{equation*}
b_{\rm cav}=\Phi(a),
\qquad
d_{\rm cav}=1,
\qquad
p_{\rm cav}=1-\Phi(a).
\end{equation*}
In particular, $p_{\rm cav}$ is strictly increasing with the hole radius $a$.
\end{proposition}

\begin{proof}
Inside the cavity, $\dist(x,D)$ is maximal at the center and equals $a$.
Since $\Phi$ is strictly decreasing, the cavity component is born at level
$\Phi(a)$. The assumption on the imaging window guarantees an exterior
component born at a strictly lower level, so the cavity class is not the
essential $H_0$ class. For every $t<1$, the material annulus, on which
$I_0=1$, separates the cavity from the exterior. The two components first
merge at $t=1$, giving the stated persistence.
\end{proof}

In particular,
\begin{equation*}
a\downarrow0
\qquad\Longrightarrow\qquad
p_{\rm cav}\downarrow0,
\end{equation*}
so a persistence cutoff acts as a cavity-resolution cutoff in this model.

For a scalar field $f$, let $\mathcal D_0(f)$ denote its zero-dimensional
sublevel persistence diagram. The bottleneck stability theorem
\cite{Cohen2007} states that
\begin{equation*}
d_B\bigl(\mathcal D_0(f),\mathcal D_0(g)\bigr)
\le\|f-g\|_\infty,
\end{equation*}
where $d_B$ denotes the bottleneck distance.

\begin{theorem}
\label{thm:barstability}
Suppose $\|I-I_0\|_\infty\le\varepsilon$. If a finite bar of
$\mathcal D_0(I_0)$ has persistence $p_0$ satisfying
\begin{equation*}
p_0>\tau+2\varepsilon,
\end{equation*}
then it is matched to a bar of $\mathcal D_0(I)$ whose persistence is greater
than $\tau$. Conversely, if a bar of $\mathcal D_0(I_0)$ has persistence
below $\tau-2\varepsilon$, its matched bar in $\mathcal D_0(I)$ remains below
$\tau$.
\end{theorem}

\begin{proof}
Under a bottleneck matching of cost at most $\varepsilon$, the birth and death
coordinates of a matched off-diagonal point each change by at most
$\varepsilon$. Its persistence therefore changes by at most $2\varepsilon$.
The first inequality also gives $p_0>2\varepsilon$, so the corresponding point
cannot be matched to the diagonal.
\end{proof}

For the annular ideal model, a sufficient condition for the cavity bar to
remain above the cutoff is therefore
\begin{equation*}
1-\Phi(a)>\tau+2\varepsilon.
\end{equation*}
This controls persistence magnitude. The additional birth-location condition
in \eqref{eq:cavity-index-set} is what allows the algorithm to interpret a
persistent bar as an interior cavity rather than as background structure.

The final binary reconstruction has the detected Betti vector whenever the
localized cavity regions are geometrically separated in the expected way.

\begin{proposition}
\label{prop:finaltopology}
Suppose $E_*$ has $m$ connected components and no bounded complementary
components. Let $\widehat C_1,\ldots,\widehat C_r$ be pairwise disjoint
connected sets compactly contained in $\operatorname{int}(E_*)$, and assume
that removing each $\widehat C_j$ does not disconnect the component of $E_*$
that contains it. Then
\begin{equation*}
\widehat D
=
E_*\setminus\bigcup_{j=1}^r\widehat C_j
\end{equation*}
satisfies
\begin{equation*}
\beta_0(\widehat D)=m,
\qquad
\beta_1(\widehat D)=r.
\end{equation*}
\end{proposition}

\begin{proof}
The removals leave the $m$ exterior components connected, so
$\beta_0(\widehat D)=m$. The complement of $\widehat D$ has one unbounded
component and exactly $r$ bounded components created by the removed cavity
sets. Proposition~\ref{prop:alexander} then gives
$\beta_1(\widehat D)=r$.
\end{proof}

The preceding results separate algorithmic consistency from physical
resolution. Proposition~\ref{prop:finaltopology} describes when the binary
set constructed from the detected envelope and cavity masks realizes the
corresponding component and cavity counts. Equality with the physical Betti
vector requires, in addition, that the physical structures survive the
resolution scale of the indicator. The offset and persistence arguments above
make this second requirement explicit in the ideal model.

\section{Numerical experiments}
\label{sec:numerics}

This section evaluates the full reconstruction pipeline on full-wave data and
then examines its resolution limits by varying the relevant geometric scales.
The physical topology is used only for evaluation; it is not supplied to the
algorithm.

\subsection{Computational setting}

The forward data are generated from the Lippmann--Schwinger equation
\eqref{eq:LS} on a uniform Cartesian grid. The volume integral is approximated
by cellwise quadrature, and the diagonal Green-kernel value is replaced by its
disk-averaged cell value. Thus all reported data are generated by the full
Lippmann--Schwinger model rather than by the Born approximation.

Unless stated otherwise, we use $k=6$, $q_0=0.5$, $M=80$ uniformly distributed
incident and observation directions, a $60\times60$ forward grid, and a
$180\times180$ reconstruction grid. In \eqref{eq:FM}, the regularization is
\begin{equation*}
\varepsilon_{\rm FM}=10^{-7}\lambda_{\max}(F_\#),
\end{equation*}
and the superlevel scan uses $400$ uniformly spaced thresholds in
$[0.001,0.999]$.

For the digital topology, foreground superlevel sets use eight-neighbor
connectivity, while the complementary background and the sublevel $H_0$
filtration use four-neighbor connectivity. This complementary convention
avoids the usual foreground/background ambiguity in two-dimensional digital
images and is consistent with the dual viewpoint in \cite{Bleile2022}. Finite
$H_0$ bars are computed by an elder-rule union--find algorithm.

The representative figures use $5\%$ noise according to \eqref{eq:noise}, and
we take $\tau=10^{-4}$ unless stated otherwise. Reconstruction accuracy is
measured by the intersection-over-union
\begin{equation*}
\operatorname{IoU}(A,B)
=
\frac{|A\cap B|}{|A\cup B|},
\end{equation*}
where $|\cdot|$ denotes area, approximated by pixel count on the reconstruction
grid.

\subsection{Representative reconstructions}

The first configuration consists of one disk and one annulus, with physical
Betti vector $\bm\beta(D)=(2,1)$. At $5\%$ noise, the fixed threshold
$I\ge0.5$ gives $(3,0)$ with IoU $0.491$. The proposed superlevel stage
selects a two-component exterior state, and the sublevel filtration contains
one significant interior bar with persistence approximately $0.192$. The
resulting detectable topology is $(2,1)$ and the final IoU is $0.776$.

\begin{figure}[t]
\centering
\includegraphics[width=0.96\textwidth]{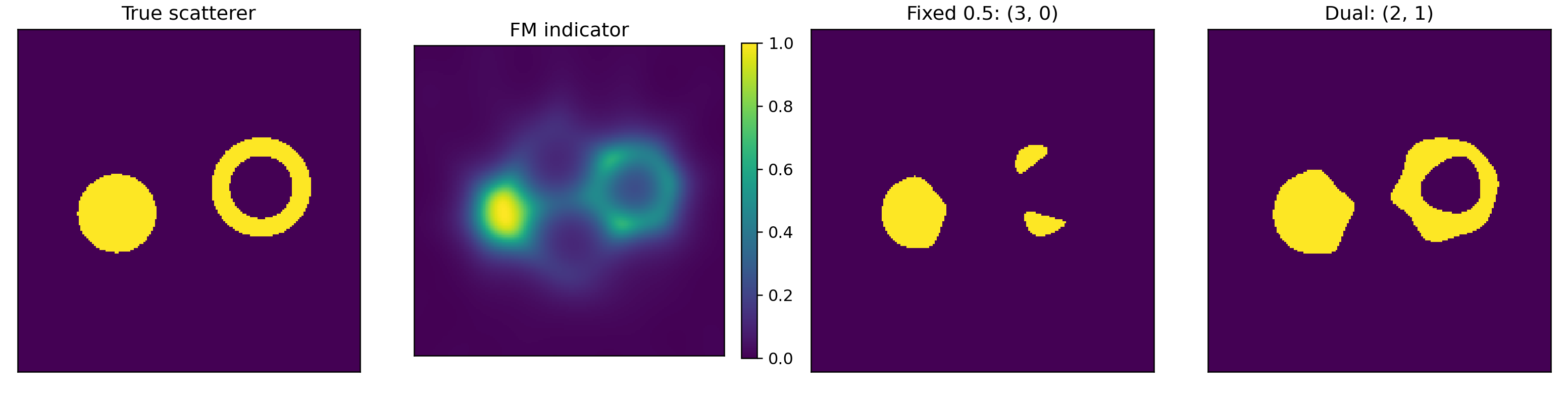}
\caption{Disk--annulus example at $5\%$ noise. From left to right: physical scatterer, normalized factorization indicator, reconstruction using the fixed threshold $0.5$, and the dual-filtration reconstruction.}
\label{fig:diskannulus}
\end{figure}

The second group contains two multi-hole configurations. The connected
double-window frame has physical topology $(1,2)$. Here the fixed threshold
already gives the correct topology, with IoU $0.665$. The dual-filtration
procedure also gives $(1,2)$; the two detected cavity bars have persistences
approximately $0.454$ and $0.453$, and the final IoU is $0.666$. This example
shows that the proposed procedure is not predicated on the failure of a
particular fixed threshold.

The more demanding configuration consists of two annuli and one disk and has
physical topology $(3,2)$. For this test, we use $k=2$, $q_0=0.5$, $M=96$
incident and observation directions, a $100\times100$ forward grid, and a
$220\times220$ reconstruction grid. The annuli have inner-to-outer radius
ratio $0.636$, and the minimum separation between distinct components is $5.0$
in the chosen units. At $5\%$ noise, the fixed threshold $0.5$ gives $(1,0)$
with IoU $0.288$. The superlevel stage identifies three exterior components,
while the global sublevel filtration contains two persistent components born
inside the filled envelope, with persistences approximately $0.0139$ and
$0.0133$. The combined reconstruction gives $(3,2)$ and IoU $0.692$.

\begin{figure}[t]
\centering
\includegraphics[width=0.96\textwidth]{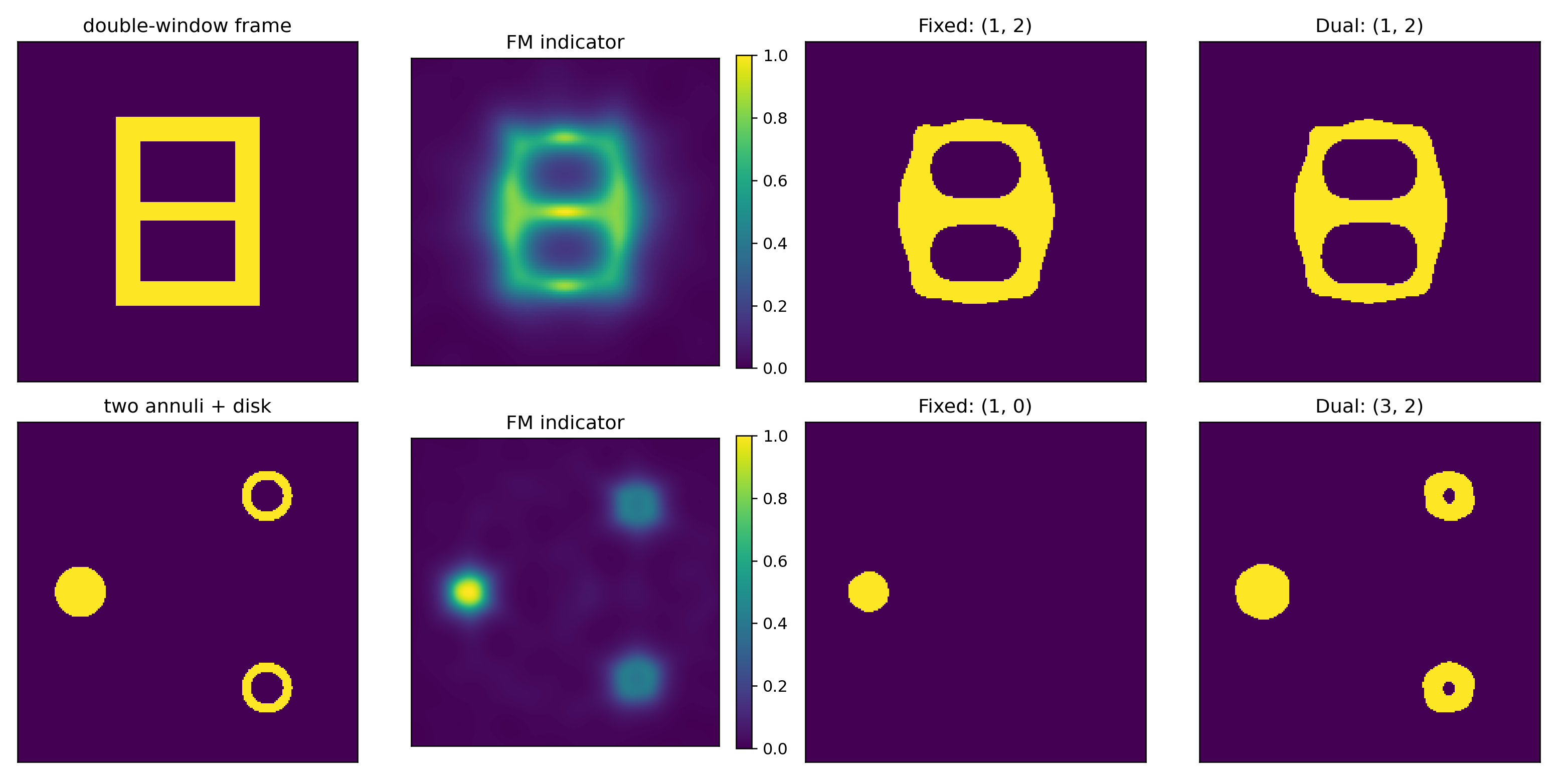}
\caption{Multiple-cavity examples at $5\%$ noise. Top: connected double-window frame. Bottom: two annuli and one disk. The latter illustrates a case in which the exterior components and cavities are most naturally identified at different intensity levels.}
\label{fig:multihole}
\end{figure}

For the two-annuli-plus-disk configuration, we also tested one noiseless
realization and six independent realizations at each of $2\%$, $5\%$, and
$10\%$ noise. For all $19$ realizations, the detected topology remained
$(3,2)$ for each of
\begin{equation*}
\tau=10^{-5},\qquad 10^{-4},\qquad 10^{-3}.
\end{equation*}
Thus all $57$ noise--cutoff combinations produced the same detectable Betti
vector. At $10\%$ noise, the smallest second-cavity persistence among these
tests was approximately $1.13\times10^{-3}$.

Table~\ref{tab:summary} summarizes the representative $5\%$-noise results.

\begin{table}[t]
\centering
\caption{Representative reconstruction results at $5\%$ noise.}
\label{tab:summary}
\small
\begin{tabular}{lccccc}
\toprule
Geometry
& $\bm\beta(D)$
& Fixed $\bm\beta$
& $\widehat{\bm\beta}_{\rm det}$
& Fixed IoU
& Dual IoU
\\
\midrule
Disk--annulus
& $(2,1)$ & $(3,0)$ & $(2,1)$ & $0.491$ & $0.776$
\\
Double-window frame
& $(1,2)$ & $(1,2)$ & $(1,2)$ & $0.665$ & $0.666$
\\
Two annuli + disk
& $(3,2)$ & $(1,0)$ & $(3,2)$ & $0.288$ & $0.692$
\\
Small-hole annulus
& $(1,1)$ & $(1,0)$ & $(1,0)$ & $0.899$ & $0.904$
\\
Close disks
& $(2,0)$ & $(1,0)$ & $(1,0)$ & $0.630$ & $0.649$
\\
\bottomrule
\end{tabular}
\end{table}

\subsection{Resolution limits and parameter sweeps}

The small-hole annulus and the close-disk configuration illustrate the two
resolution mechanisms described by the ideal model. In the small-hole case,
the physical topology is $(1,1)$, but the factorization indicator contains no
significant interior sublevel bar inside the selected envelope, so the method
reports detectable topology $(1,0)$. In the close-disk case, the physical
topology is $(2,0)$, but the dominant bright region is connected, and the
method again reports $(1,0)$. These outputs are interpreted as resolution
limits of the available indicator rather than as attempts to enforce a known
physical Betti vector.

\begin{figure}[t]
\centering
\includegraphics[width=0.96\textwidth]{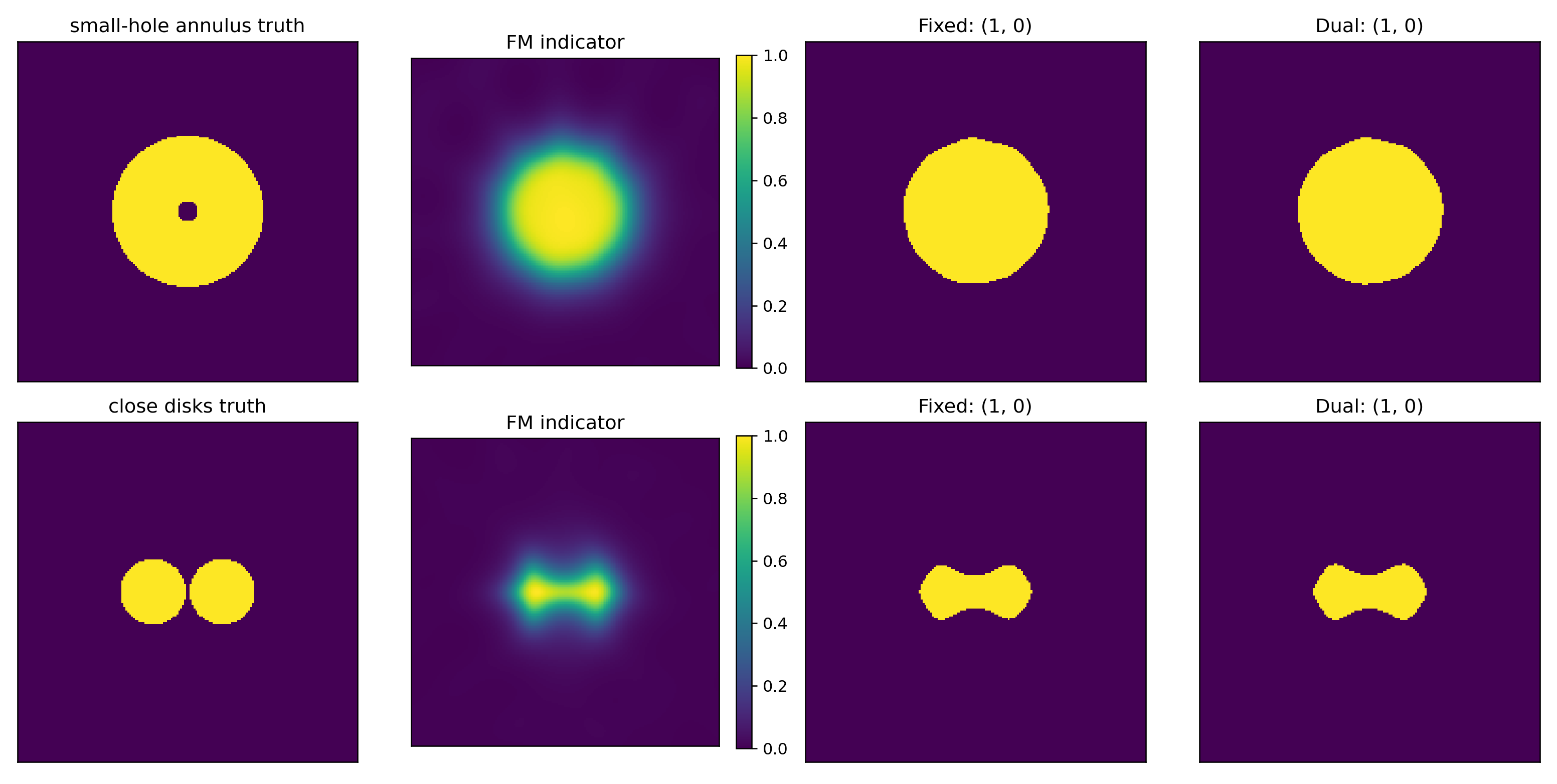}
\caption{Resolution-limit examples at $5\%$ noise. Top: the small physical cavity is not represented by a significant interior persistence bar. Bottom: two nearby disks are merged in the factorization indicator.}
\label{fig:limits}
\end{figure}

To determine whether these are isolated examples or part of systematic
transitions, we vary the relevant geometric scales. For two equal disks of
radius $0.25$, the boundary gap is sampled at
\begin{equation*}
g=0.02,\ 0.04,\ 0.07,\ 0.10,\ 0.14,\ 0.20,\ 0.30,\ 0.45.
\end{equation*}
For both noiseless data and $5\%$ noise, the detected component count is
$\widehat\beta_0=1$ for $g\le0.20$ and becomes $2$ at $g=0.30$ and $0.45$.

For an annulus with outer radius $0.58$, the inner radius is sampled at
\begin{equation*}
a=0.035,\ 0.055,\ 0.075,\ 0.095,\ 0.120,\ 0.160,\ 0.210,\ 0.270.
\end{equation*}
No cavity is detected for $a\le0.095$. At $a=0.120$, the first cavity bar
appears, with persistence approximately $9.86\times10^{-3}$ for noiseless
data and $2.09\times10^{-3}$ at $5\%$ noise. The persistence then increases
with the hole size; at $a=0.270$, it is approximately $0.375$ and $0.387$ for
noiseless and $5\%$ noisy data, respectively.

\begin{figure}[t]
\centering
\includegraphics[width=0.93\textwidth]{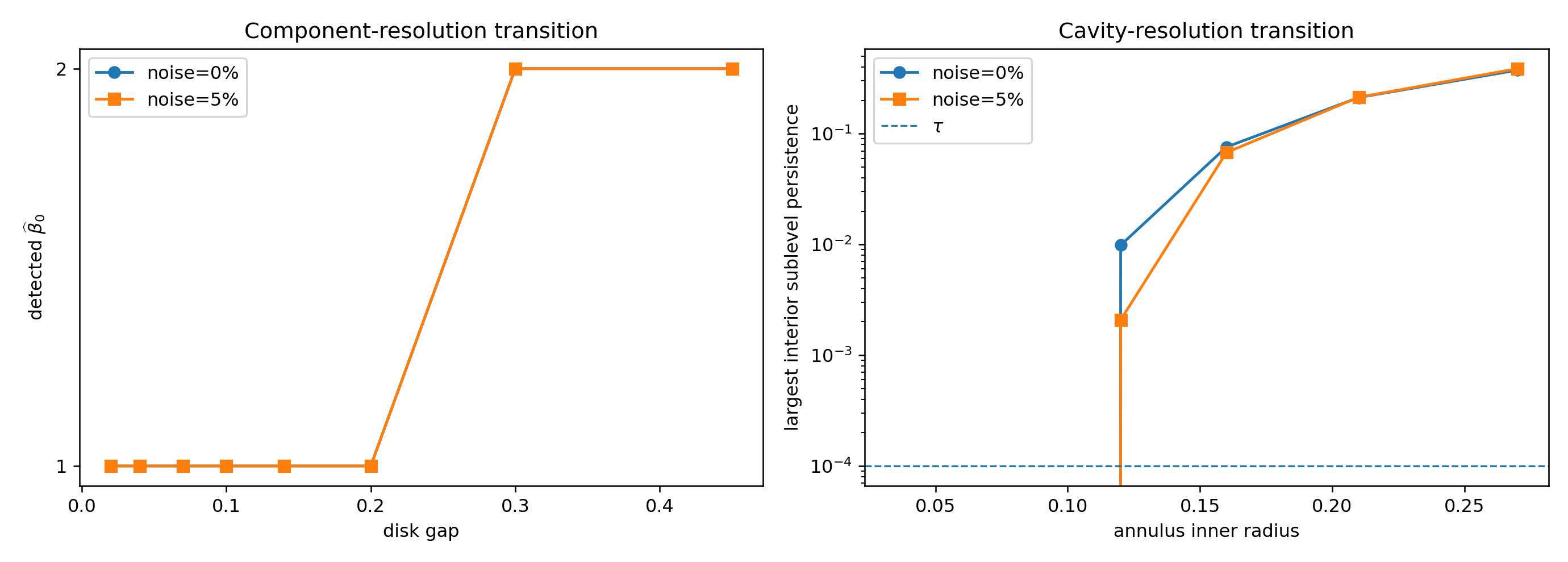}
\caption{Topological resolution sweeps. Left: detected component count as the gap between two disks increases. Right: largest interior sublevel persistence as the annular hole radius increases. The dashed line denotes the cutoff $\tau=10^{-4}$.}
\label{fig:phasediagram}
\end{figure}

The transition locations in Figure~\ref{fig:phasediagram} should not be
identified numerically with the reach values of Section~\ref{sec:analysis},
because the factorization indicator is not the distance-profile model.
Nevertheless, the same qualitative mechanism is visible: increasing the
separation eventually prevents superlevel merging, while increasing the cavity
size produces a longer and more robust interior sublevel bar.

These experiments also clarify the role of persistent homology relative to
fixed-threshold homology. A suitable fixed threshold can recover the correct
topology, as in the double-window frame. The purpose of the filtration is not
to improve the homology computation at a prescribed level, but to examine the
available topological evidence over a range of intensity scales, estimate the
detectable Betti vector without prescribing it, and allow the exterior and
cavity structures to be reconstructed at different levels. The resolution
sweeps also show that persistent postprocessing does not recover features that
are absent from the underlying indicator.

\section{Conclusion}
\label{sec:conclusion}

We have developed a dual-filtration reconstruction for qualitative acoustic
inverse scattering. Superlevel sets determine the exterior component
structure, while finite sublevel $H_0$ persistence bars are filtered by their
birth locations relative to a filled exterior envelope to determine the
cavities. The true Betti vector is not an input, and different topological
features can be localized at different intensity levels. Planar duality and
the coarea formula provide the corresponding topological and energetic
interpretations of the construction.

An ideal distance-profile model makes the principal resolution limits
explicit: positive reach controls topology preservation of geometric offsets,
while the two-disk and annular examples describe component merging and cavity
filling. For an annulus, cavity persistence increases with the hole radius,
and persistence stability gives a perturbation-dependent sufficient condition
for retaining the cavity bar. The full-wave parameter sweeps exhibit the same
qualitative transitions for factorization indicators. The method should
therefore be interpreted as estimating the topology detectable from the
available indicator, rather than forcing agreement with physical features
that are not resolved.

\end{document}